\documentclass[12pt,oneside,reqno]{amsart}
\usepackage{latexsym,amsmath,amsfonts,amsthm}
\usepackage[usenames]{color}
\usepackage[usenames,dvipsnames,svgnames,table]{xcolor}
\usepackage[all]{xy}  

\usepackage{upgreek}

\usepackage[]{graphicx}
\usepackage{wrapfig}

\DeclareSymbolFont{rmlargesymbols}{OMX}{mdbch}{m}{n}
\DeclareMathSymbol{\rmintop}{\mathop}{rmlargesymbols}{82}
\DeclareMathSymbol{\rmointop}{\mathop}{rmlargesymbols}{72}
\DeclareMathSymbol{\rmsumop}{\mathop}{rmlargesymbols}{80}
\DeclareMathSymbol{\rmunionop}{\mathop}{rmlargesymbols}{83}
\DeclareMathSymbol{\rmintersectop}{\mathop}{rmlargesymbols}{84}
\DeclareMathSymbol{\rmtensorop}{\mathop}{rmlargesymbols}{79}
\DeclareMathSymbol{\rmdirectsumop}{\mathop}{rmlargesymbols}{77}

\usepackage{mathrsfs}

\usepackage{mathptmx}
\usepackage{bookman}
\usepackage{relsize} 

\usepackage[font=small,labelfont=bf]{caption}
\usepackage{sidecap}
\usepackage{float}

\input xy
\xyoption{all}

\usepackage{amsmath}
\usepackage{amsfonts}
\usepackage{hyperref}
\usepackage{amssymb}
\usepackage[]{graphicx}
\usepackage{upgreek}
\usepackage{mathptmx}

\usepackage{amssymb}
\usepackage{amscd}
\usepackage{array}
\usepackage{verbatim}

\theoremstyle{definition}
\newtheorem{thm}{Theorem}[section]
\newtheorem{lem}[thm]{Lemma}

\newtheorem{dfn}[thm]{Definition}
\newtheorem{cor}[thm]{Corollary}

\newtheorem{rmk}[thm]{Remark}

\newtheorem{qst}[thm]{Question}

\newcommand{\bei}{\begin{itemize}}
\newcommand{\eei}{\end{itemize}}

\newcommand{\QED}{\rule{0.4em}{2ex}}

\def\conj#1{\overline{#1}}

\def\vartau{\uptau}

\def\T{\bold T}

\def\ccite#1{\textcolor{Red}{\cite{#1}}}

 \def\cequiv{ {\text{\ \Large$\sim_{{}_{\!\!\!\!\!\!\!\!c}}$\ \,}}  }

\numberwithin{equation}{section}

\begin{document}

\title[ M\lowercase{odular} I\lowercase{mages} O\lowercase{f} A\lowercase{pproximately} C\lowercase{entral} P\lowercase{rojections} ]
{\Large \rm M\lowercase{odular} I\lowercase{mages} O\lowercase{f}  \\  A\lowercase{pproximately} C\lowercase{entral} P\lowercase{rojections}
}
\author{S\lowercase{amuel} G.\,W\lowercase{alters} \\ 
{\Tiny U\lowercase{niversity of} N\lowercase{orthern} B\lowercase{ritish} 
C\lowercase{olumbia}}}
\address{Department of Mathematics and Statistics, University  of Northern B.C., Prince George, B.C. V2N 4Z9, Canada.}
\email[]{samuel.walters@unbc.ca}
\thanks{\scriptsize \LaTeX\ File: ModularImagesACprojectionsArXiv.tex}
\urladdr{web.unbc.ca/~walters \ or \  hilbert.unbc.ca/}

\begin{abstract}
It is shown that for any approximately central (AC) projection $e$ in the Flip orbifold $A_\theta^\Phi$ (of the irrational rotation C*-algebra $A_\theta$), and any modular automorphism $\alpha$ (arising from SL$(2,\mathbb Z)$), the AC projection $\alpha(e)$ is centrally Murray-von Neumann equivalent to one of the projections $e,\ \sigma(e),\ \kappa(e),\ \kappa^2(e),$ $\sigma\kappa(e),\ \sigma\kappa^2(e)$ in the $S_3$-orbit of $e,$ where $\sigma, \kappa$ are the Fourier and Cubic transforms of $A_\theta$. (The equivalence being implemented by an approximately central partial isometry in $A_\theta^\Phi$.)  For smooth automorphisms $\alpha,\beta$ of the Flip orbifold $A_\theta^\Phi$, it is also shown that if $\alpha_*=\beta_*$ on $K_0(A_\theta^\Phi),$ then $\alpha(e) \cequiv \beta(e)$ are centrally equivalent for each AC projection $e$.
\end{abstract}

\maketitle

{\Large \section{Main Results}}

The irrational rotation C*-algebra $A_\theta$ is generated by unitary operators $U,V$ satisfying the Heisenberg commutation
\[
VU = e^{2\pi i\theta} UV
\]
where $\theta$ is an irrational number. The Flip $\Phi,$ Fourier transform $\sigma,$ and Cubic transform $\kappa,$ are the canonical automorphisms of $A_\theta$ defined by
\begin{align*}
 \Phi(U) &= U^{-1},  &  \sigma(U) &= V^{-1}, & \kappa(U) &= e(-\tfrac\theta2)U^{-1}V,
\\
\Phi(V) &= V^{-1}, &  \sigma(V) &= U,	 & \kappa(V) &= U^{-1},
\end{align*}
where we shall sometimes use the conventional notation $e(\theta) := e^{2\pi i\theta}$. The Flip orbifold is the fixed point C*-subalgebra of $A_\theta$ given by $A_\theta^\Phi = \{x\in A_\theta: \, \Phi(x) = x\}$. The Flip was studied extensively, e.g. in \ccite{BEEKa} \ccite{BEEKb} \ccite{SW-JLMS}, and it is known that the Flip orbifold is approximately finite dimensional for irrational $\theta$ (\ccite{BK} and \ccite{SW-CMP}).
\medskip

The well-known modular action of the group SL$(2,\mathbb Z)$ on $A_\theta$ (due to Watatani \ccite{YW} and Brenken \ccite{BB}) is given by associating to the integral matrix 
$\left[\smallmatrix a & c \\ b & d \endsmallmatrix\right]$ (with $ad-bc=1$) the automorphism $\alpha$ defined by
\begin{equation}\label{modaut}
\alpha(U) = e(\tfrac\theta2 ab)\, U^a V^b, \qquad
\alpha(V) = e(\tfrac\theta2 cd)\, U^c V^d.
\end{equation}
We call these modular automorphisms of the rotation algebra. It is straightforward to check that the unitaries on the right sides of these equalities satisfy the commutation relation $VU = e(\theta) UV,$ guaranteeing that the automorphism $\alpha$ exists by the universal property of this relation.
\medskip

It's easy to check that all these modular automorphisms commute with the Flip, being given by minus the identity matrix $\left[\smallmatrix -1 & 0 \\ 0 & -1 \endsmallmatrix\right]$. Therefore, the modular group SL$(2,\mathbb Z)$ also acts on the Flip orbifold $A_\theta^\Phi,$ leaving it invariant.
\medskip

\begin{dfn}
A projection $e$ in a C*-algebra $A$ is called {\bf approximately central} (or AC, for short) if $e=e_n$ is a sequence depending on an integer parameter $n$ such that
\[
\lim_{n\to \infty} \|x e_n - e_n x\| \ = \ 0
\]
for each $x\in A$.  Two AC projections $e=e_n$ and $f=f_n$ are said to be {\bf centrally equivalent} or {\bf AC-equivalent} in $A$ if $\forall \epsilon > 0$ and each finite subset $F$ of $A,$ there exists a partial isometry $u \in A$ such that $uu^* = e_n, \ u^*u = f_n$ (for large enough $n$) and $\|x u - u x\| < \epsilon$ for each $x\in F$. We use the notation
\[
e \cequiv f
\]
to indicate central equivalence of $e$ and $f$ in $A$. 
\end{dfn}

(Of course, this last condition is just usual Murray-von Neumann equivalence with the added condition that the partial isometry, or unitary, is approximately central in the algebra $A$.)
\medskip

\begin{dfn}
A projection $e$ in the rotation algebra is called {\bf smooth} when it belongs to the canonical smooth dense *-subalgebra $A_\theta^\infty$ consisting of rapidly decreasing Laurent series $\rmsumop a_{mn}U^mV^n$. An automorphism is called {\bf smooth} if it maps $A_\theta^\infty$ onto itself.
\end{dfn}
\medskip

Of course, all modular automorphisms are smooth. The smooth Flip orbifold $A_\theta^{\infty,\Phi} = A_\theta^\infty \cap A_\theta^\Phi$ is also a dense *-subalgebra of $A_\theta^\Phi$ and is closed under the holomorphic functional calculus (as is the case for $A_\theta^\infty$ and $A_\theta$).
\medskip

We can now state the main results.

\begin{thm}\label{mainthm}
Let $\alpha$ be a modular automorphism of the irrational rotation C*-algebra $A_\theta$ arising from the canonical action of SL$(2,\mathbb Z)$. For any approximately central projection $e$ in the Flip orbifold $A_\theta^\Phi,$ the projection $\alpha(e)$ is centrally equivalen in $A_\theta^\Phi$ to one of the projections in its $S_3$-orbit
\begin{equation}\label{orbite}
e, \quad \sigma(e), \quad \kappa(e), \quad \kappa^2(e), \quad \sigma\kappa(e), \quad \sigma\kappa^2(e).
\end{equation}
\end{thm}
\bigskip

The symmetric group $S_3$ is here envisaged as generated by $\sigma$ and $\kappa$ subject to the relation $\kappa\sigma = \sigma\kappa^2$ (and $\sigma^2 = \kappa^3 = 1$). Some of the projections in the $S_3$-orbit of $e$ could be centrally equivalent to one another. Indeed, they can all be according to the following result.
\medskip

\begin{thm}\label{prop}
Let $\alpha$ be a smooth automorphism of the irrational rotation C*-algebra $A_\theta$ commuting with the Flip $\Phi$. Then for any AC smooth projection $g$ in $A_\theta$ such that $g \Phi(g) = 0,$ the projections $e = g + \Phi(g)$ and $\alpha(e)$ are centrally equivalent in $A_\theta^\Phi$. In particular, the $S_3$-orbit projections of $e$ \eqref{orbite} are all centrally equivalent in $A_\theta^\Phi$.
\end{thm}
\medskip

(Note that the automorphism in this theorem is not required to be modular.) In a forthcoming paper \ccite{SWb} we construct an AC Powers-Rieffel projection such that its six $S_3$-orbit projections \eqref{orbite} are pairwise not centrally equivalent (Theorem 1.4 of \ccite{SWb}). 
\medskip

The orbifold $A_\theta^{\infty,\Phi}$ is known to have four basic unbounded traces denoted  
$\phi_{jk}$ (see equation \eqref{traces}). The following result shows that smooth automorphisms act on them in a very specific manner.

\begin{thm}\label{mainB}
Let $\beta$ be an automorphism of the smooth Flip orbifold $A_\theta^{\infty,\Phi}$. One has the linear combinations
\[
\phi_{00}\beta = \phi_{00} + b\phi_{01} +  c\phi_{10} + d \phi_{11} 
\]
\[
\phi_{jk}\beta = b_{jk}\phi_{01} +  c_{jk} \phi_{10} + d_{jk} \phi_{11} 
\]
for $jk\in\{01,10,11\},$ where $b, c, d, b_{jk}, c_{jk}, d_{jk} \in \tfrac12\mathbb Z$ are half-integers. In particular, the latter gives a group representation $\text{Aut}(A_\theta^{\infty,\Phi}) \to \text{SL}_{\pm}(3,\mathbb R)$.
\end{thm}
\medskip

(The last representation is discussed in Remark 3.6 below.) From this we are able to say when an AC projection is centrally equivalent to a smooth automorphic image.

\begin{thm}\label{mainC}
Let $\beta$ be a smooth automorphism of the Flip orbifold that induces the identity on $K_0(A_\theta^{\infty,\Phi})$. Then
\[
\beta(e) \cequiv e
\]
are centrally equivalent in $A_\theta^{\Phi}$ for any approximately central projection $e$ in $A_\theta^{\infty,\Phi}$. 
\end{thm}

\begin{cor}
Let $\alpha, \beta$ be smooth automorphisms of the Flip orbifold $A_\theta^{\infty,\Phi}$ such that $\alpha_* = \beta_*$ on $K_0(A_\theta^{\Phi})$ (e.g., if they are homotopic). Then 
\[
\alpha(e) \cequiv \beta(e)
\]
in $A_\theta^{\Phi}$ for any approximately central projection $e$ in $A_\theta^{\infty,\Phi}$.
\end{cor}

\begin{rmk}
The canonical toral automorphism $\gamma_1(U) = -U,\  \gamma_1(V)=V$ does not induce the identity on $K_0(A_\theta^{\Phi})$ of the orbifold (see \eqref{gamma1Q}), but to show that $\gamma_1(e)$ and $e$ are not centrally equivalent for some AC projection $e$ in the Flip orbifold is not trivial. This, however, is shown to follow from \ccite{SWb} for an AC Powers-Rieffel projection $e$.  (The same is also shown to occur for the Fourier and Cubic transforms on the Flip orbifold.)
\end{rmk}

\medskip

We recall the basic unbounded $\Phi$-traces $\phi_{jk}$ on the rotation algebra $A_\theta$ are defined by 
\begin{equation}\label{traces}
\phi_{jk}(U^mV^n)\ =\ e(-\tfrac{\theta}2 mn)\,\updelta_2^{m-j} \updelta_2^{n-k}
\end{equation}
for $jk = 00, 01,10,11,$ where $\updelta_2^n$ is the divisor delta function defined to be 1 if $n$ is even, and 0 if $n$ is odd. 
\medskip

By definition, a $\Phi$-trace on $A_\theta$ is a linear functional $\phi$ (usually discontinuous) defined on the canonical smooth subalgebra $A_\theta^\infty$ such that
\[
\phi(xy) = \phi(\Phi(y) x)
\]
for all smooth elements $x,y \in A_\theta^\infty$. It is well-known (e.g., from \ccite{SW-JLMS}) that the complex vector space of all $\Phi$-traces on $A_\theta^\infty$ is 4-dimensional and has the four functionals \eqref{traces} as basis.  It is easy to see that $\phi\alpha$ is a $\Phi$-trace for each $\Phi$-trace $\phi$ and each smooth automorphism $\alpha$ commuting with the Flip. So, in particular, $\phi\alpha$ is a linear combination of the $\phi_{jk}$. As $\Phi$-traces on $A_\theta$, the maps $\phi_{jk}$ restrict to actual (unbounded) trace maps on the smooth Flip orbifold, and hence induce real-valued group morphisms on $K_0(A_\theta^{\infty,\Phi}) = K_0(A_\theta^{\Phi})$.
\medskip

The canonical trace $\vartau$ and the unbounded traces $\phi_{jk}$ together form the Connes-Chern character map for the Flip orbifold given by
\begin{equation} 
\T : K_0(A_\theta^\Phi) \to \mathbb R^5, \qquad
\T(x) = (\vartau(x); \phi_{00}(x), \phi_{01}(x), \phi_{10}(x), \phi_{11}(x))
\end{equation}
since each positive $K_0$ class $x$ has a smooth representative. In \ccite{SW-JLMS} (Proposition 3.2) this map was shown to be injective for irrational $\theta$.\footnote{In \cite{SW-JLMS} we worked with the crossed product algebra $A_\theta \times_\Phi \mathbb Z_2$, but since this is strongly Morita equivalent to the fixed point algebra, the injectivity follows for the fixed point subalgebra.}  The ranges of the traces $\phi_{jk}$ on projections are known to be half-integers: $\tfrac12\mathbb Z$, while the canonical trace has range $(\mathbb Z + \mathbb Z\theta) \cap [0,1]$ on projections in $A_\theta^\Phi$.
\medskip

If $m,n$ is any pair of integers, we shall write $\phi_{mn} = \phi_{\conj m\,\conj n}$ where $\conj m = 0, 1$ according to whether $m$ is even or odd (respectively).
\medskip

The zeroth cyclic cohomology (see \ccite{AC}) $HC^0(A_\theta^{\infty,\Phi})$ of the Flip orbifold is a 5-dimensional complex vector space spanned by the (bounded) canonical trace $\tau$ and the four unbounded traces $\phi_{jk}$. Indeed, this can be seen in much the same way as how the zeroth cyclic cohomology was calculated for the Fourier transform in \ccite{SWChern} (Proposition 2.2 and Theorem 2.3). Thus $HC^0(A_\theta^{\infty,\Phi}) \cong \mathbb C^5.$

\begin{cor}\label{coro}
A smooth automorphism of the Flip orbifold that induces the identity map on $K_0(A_\theta^{\Phi}) = \mathbb Z^6$ necessarily induces the identity map on the zeroth cyclic cohomology group $HC^0(A_\theta^{\infty,\Phi}) = \mathbb C^5$.
\end{cor}
\medskip

\noindent{\bf Acknowledgement.}
This paper and \ccite{SWb} were written at about the time the author retires, he is therefore most grateful to his home institution of 26 years, the University of Northern British Columbia, for many years of research support. The author also expresses his nontrivial gratitude to the many referees who made helpful review reports over the years (including critical ones).

{\Large \section{Proof Of Theorem \ref{mainthm}}}

We shall let $\alpha$ be a fixed modular automorphism, as in equation \eqref{modaut}. We require the following lemma.

\medskip

\begin{lem}\label{lemma}
Let $\alpha$ be a modular automorphism of the rotation algebra $A_\theta$ induced by the integral matrix $\left[\smallmatrix a & c \\ b & d \endsmallmatrix\right]$ (where $ad-bc=1$). Then
\[
\phi_{00}\alpha = \phi_{00}
\]
and its action on the remaining unbounded $\Phi$-traces is given as follows
\[
\text{For even } a: \quad \phi_{01}\alpha = \phi_{10}, \quad \phi_{10}\alpha = \phi_{d,1}, \quad \phi_{11}\alpha = \phi_{1-d, 1}
\]
\[
\text{For even } b: \quad \phi_{01}\alpha = \phi_{c,1}, \quad \phi_{10}\alpha = \phi_{10}, \quad \phi_{11}\alpha = \phi_{1-c, 1}
\]
\[
\text{For even } c: \quad \phi_{01}\alpha = \phi_{01}, \quad \phi_{10}\alpha = \phi_{1,b}, \quad \phi_{11}\alpha = \phi_{1,1-b}
\]
\[
\text{For even } d: \quad \phi_{01}\alpha = \phi_{1,a}, \quad \phi_{10}\alpha = \phi_{01}, \quad \phi_{11}\alpha = \phi_{1,1-a}
\]
\end{lem}
\begin{proof} A straightforward computation verifies that 
\[
(U^k V^\ell)^m = e(\tfrac{\theta}2 k\ell m(m-1))\ U^{km} V^{\ell m}.
\]
In view of \eqref{modaut}, this gives
\[
\alpha(U^m V^n) = e(\tfrac{\theta}2[abm^2 + cdn^2])\ e(\theta bcmn) \ U^{am+cn} V^{bm+dn}.
\]
From this it follows that $\alpha$ transforms the $\Phi$-trace functionals according to
\[
\phi_{jk}\alpha (U^mV^n)\ 
=\ e(-\tfrac{\theta}2 mn)\, \updelta_2^{am+cn-j} \updelta_2^{bm+dn-k}.
\]
Since $ad-bc=1$, one of the integers $a,b,c,d$ has to be even, so it will be convenient to express $\phi_{jk}\alpha$ in terms of the basic $\Phi$-traces $\phi_{jk}$ by considering the four cases where one of these integers is even. 
\medskip

{\bf Case 1:} $a$ is even. Here, $b,c$ are odd so they can be replaced by 1's in the delta functions, and of course $a$ disappears, so we have
\[
\phi_{jk}\alpha (U^mV^n)\ =\ e(-\tfrac{\theta}2 mn)\, \updelta_2^{m+dn-k} \updelta_2^{n-j}
=\ e(-\tfrac{\theta}2 mn)\, \updelta_2^{m+dj-k} \updelta_2^{n-j}
\]
(since from the second delta, $n$ can be replaced by $j$ in the first delta), so 
\[
\phi_{jk}\alpha = \phi_{k-dj, j}		\tag{$a$ even}
\]
(where, as we noted, the subscript index $k-dj$ is reduced mod 2). Explicitly,
\[
\phi_{00}\alpha = \phi_{00}, \quad \phi_{01}\alpha = \phi_{10}, \quad
\phi_{10}\alpha = \phi_{d,1}, \quad \phi_{11}\alpha = \phi_{1-d, 1}.
\]
In similar vain we proceed to deal with the remaining cases.

\medskip

{\bf Case 2:} $b$ is even. Here $a,d$ are odd so they can be replaced by 1's so we have
\[
\phi_{jk}\alpha (U^mV^n)\ =\ e(-\tfrac{\theta}2 mn)\, \updelta_2^{m+ck-j} \updelta_2^{n-k}
\]
from which
\[
\phi_{jk}\alpha = \phi_{j-ck, k}		\tag{$b$ even}
\]
\medskip

{\bf Case 3:} $c$ is even. Here $a,d$ are odd so they can be replaced by 1's so we have
\[
\phi_{jk}\alpha (U^mV^n)\ 
= e(-\tfrac{\theta}2 mn)\, \updelta_2^{m-j} \updelta_2^{n+ bj-k}
\]
from which
\[
\phi_{jk}\alpha = \phi_{j, k-bj}.		\tag{$c$ even}
\]
\medskip

{\bf Case 4:} $d$ is even. Here $b,c$ are odd, so
\[
\phi_{jk}\alpha (U^mV^n)\ 
=\ e(-\tfrac{\theta}2 mn)\, \updelta_2^{m-k} \updelta_2^{n + ak-j} 
\]
from which
\[
\phi_{jk}\alpha = \phi_{k, j-ak}.		\tag{$d$ even}
\]
From each of these four cases we obtain the results stated in the lemma.
\end{proof}

\medskip

Note that the actions on the three unbounded traces $\phi_{01}, \phi_{10}, \phi_{11}$ give rise to all their cyclic permutations, exhausting the symmetric group $S_3$ on these three objects. For instance, in the first case where $a$ is even, if $d$ is also even we have
\[
\phi_{01}\alpha = \phi_{10}, \quad \phi_{10}\alpha = \phi_{01}, \quad \phi_{11}\alpha = \phi_{11}
\]
so $\alpha$ acts like a transposition on the first two (like the Fourier transform), and when $d$ is odd we have
\[
\phi_{01}\alpha = \phi_{10}, \quad
\phi_{10}\alpha = \phi_{11}, \quad \phi_{11}\alpha = \phi_{01}
\]
so $\alpha$ induces a cyclic permutation (like the inverse Cubic transform $\kappa^{-1}$). One gets the other transpositions and cyclic permutations by similarly examining the other cases for $b,c,d$. 
\medskip

This shows that modular automorphisms act on $\phi_{01}, \phi_{10}, \phi_{11}$ simply by  permutations, and that all permutations arise from the symmetric group $S_3$ on three letters generated by the order 2 element $\sigma$ (Fourier transform on the Flip orbifold) and the order 3 element $\kappa$ (Cubic transform) satisfying the commutation $\kappa\sigma = \sigma\kappa^2$.

\medskip

Having fixed the automorphism $\alpha$, by Lemma \ref{lemma} choose $\beta$ in the symmetric group
\[
S_3 = \{ I, \sigma, \kappa, \kappa^2, \sigma\kappa, \sigma\kappa^2\}
\]
such that the $\gamma := \beta^{-1}\alpha$ fixes all the unbounded traces: $\phi_{jk}\gamma = \phi_{jk}$. It follows that $\T\gamma(x) = \T(x)$ on $K_0$ and hence $\gamma_* = id_*$ is the identity on $K_0(A_\theta^\Phi)$.
\medskip

Now suppose $[P]$ is a class from any given (but fixed) basis for $K_0(A_\theta^\Phi)$ where $P$ is a projection in $A_\theta^\Phi$. The cutdown of $P$ by the projection $\gamma(e)$ is the projection $\chi(\gamma(e) P \gamma(e))$ (defined since $e$ is approximate central), where $\chi$ is the characteristic function of the interval $[\tfrac12,\infty)$.  Writing $\phi = \tau, \phi_{jk}$ for short, the topological invariant of this cutdown is  
\[
\phi \chi(\gamma(e)P\gamma(e)) = \phi\gamma \chi(e \gamma^{-1}(P) e)
= \phi \chi(e \gamma^{-1}(P) e)
\]
since $\phi\gamma = \phi$. As $[\gamma^{-1}(P)] = [P]$ in $K_0(A_\theta^\Phi),$ they are unitarily equivalent: $\gamma^{-1}(P) = uPu^*$ for some Flip invariant unitary $u$ (by \ccite{SW-JLMS}, Theorem 5.3 or Corollary 5.6). Hence,
\begin{equation*}
\phi \chi(\gamma(e) P \gamma(e)) 
= \phi \chi(e uPu^* e) 
= \phi( u\chi(eP e) u^* )
= \phi( \chi(eP e) ) 
\end{equation*}
\vskip5pt

\noindent using the trace property of $\phi$ (and $u$ being Flip invariant). This shows that the cutdown projections $\chi(\gamma(e) P \gamma(e))$ and $\chi(eP e)$ have the same canonical trace and same unbounded traces so that their Connes-Chern characters $\T$ are equal. Therefore, since $\T$ is injective on $K_0$, their classes in $K_0(A_\theta^\Phi)$ are equal.  By Kishimoto's\footnote{See Remark 2.9 in \ccite{AK} where it is noted that his Theorem 2.1 applies equally to simple AT-algebras of real rank zero, which include the irrational rotation C*-algebra, and its canonical orbifolds under the canonical automorphisms of order 2, 3, 4, and 6. Further, the $K_1$ part of his theorem trivially holds since the $K_1$ group of the Flip orbifold is 0, being an AF-algebra.}Theorem 2.1 \ccite{AK}, the projections $\gamma(e)$ and $e$ are centrally equivalent $\gamma(e) \cequiv e$, so that $\gamma(e) = \beta^{-1}\alpha(e) = vev^*$ for some approximately central unitary $v$ in the Flip orbifold. Applying $\beta,$ we see that $\alpha(e) \cequiv \beta(e)$ are centrally equivalent in $A_\theta^\Phi$. Since $\beta$ is one of the six basic automorphisms mentioned earlier, it follows that $\alpha(e)$ is centrally equivalent to one of the projections in the $S_3$-orbit of $e$:
\begin{equation}
e, \quad \sigma(e), \quad \kappa(e), \quad \kappa^2(e), \quad \sigma\kappa(e), \quad \sigma\kappa^2(e).
\end{equation}
In particular, for any modular automorphism $\alpha$, the AC projection $\alpha(e)$ is centrally equivalent to one of these projections in the Flip orbifold. This proves Theorem \ref{mainthm}.
\medskip

We now prove Theorem \ref{prop}.

\begin{proof} (Of Theorem \ref{prop}.) One can construct AC projections $g$ in the irrational rotation C*-algebra $A_\theta$ orthogonal to their Flip: $g\Phi(g) = 0$ (see \ccite{SWmathscand}). The resulting Flip-invariant projection $e = g + \Phi(g)$ is also approximately central in $A_\theta$. Now fix a class $[P]$ from any given (but fixed) basis for $K_0(A_\theta^\Phi)$ where $P$ is a (smooth) projection in $A_\theta^\Phi$. Then since $\alpha^{-1}(P)$ and $P$ have the same trace, they are unitarily equivalent in the rotation algebra $A_\theta$ (see \ccite{MR}, Corollary 2.5, based on the highly nontrivial results of \ccite{PV} and \ccite{MR1981}). So let $w \in A_\theta$ be a unitary such that $\alpha^{-1}(P) = wPw^*$.
Since $\tau\alpha = \tau,$ we have
\[
\tau \chi(\alpha(e) P \alpha(e) ) = \tau\alpha \chi( e \alpha^{-1}(P) e) =
\tau \chi( e wPw^* e) = \tau (w\chi( e P e)w^*) = \tau (\chi( e P e)) 
\]
since the projections $\chi( e wPw^* e)$ and $w\chi( e P e)w^*$ are close (as $e$ is AC) so therefore have the same trace. 

Next, let $Q = \alpha^{-1}(P)$ and note that from the norm approximation $gQ \Phi(g) \approx 0$ (where $a \approx b$ is short for $\|a-b\|\to 0$), we have
\[
e Q e = (g + \Phi(g)) Q (g + \Phi(g)) \approx gQg + \Phi(g)Q \Phi(g) = gQg + \Phi(gQg)
\]
is a sum of two orthogonal positive elements, each norm-close to their squares (since $g$ is AC), thus giving
\[
\chi(e Q e) \approx \chi(gQg) + \Phi\chi(gQg)
\]
where both sides are $\Phi$-invariant projections. On the right side one has $\phi \chi(g Q g) =  0$ for each $\Phi$-trace $\phi$ (and $\phi \Phi\chi(gQg) = 0$). To see this, we observe that if $f = h + \Phi(h),$ where $h \Phi(h) = 0,$ then $\phi(f) = 0$ for each $\phi$ since $\phi(h) = \phi(hh) = \phi(\Phi(h)h) = 0$. Therefore, $\phi\chi(e Q e) = 0$ and we have
\[
\phi_{jk}\chi(\alpha(e) P \alpha(e)) = \phi_{jk} \alpha \chi(e Q e) = 0
\]
since $\phi_{jk} \alpha$ is also a $\Phi$-trace ($\alpha$ being $\Phi$-invariant). This shows that the projections $\chi(\alpha(e) P \alpha(e))$ and $\chi(e P e)$ have the same canonical trace and both have vanishing unbounded traces, so they yield the same class in $K_0(A_\theta^\Phi),$ and thus, by Kishimoto's Theorem, $\alpha(e)$ and $e$ are centrally equivalent in $A_\theta^\Phi$ (bearing in mind the comment in footnote 2 regarding $K_1=0$).
\end{proof}

{\Large \section{Smooth Orbifold Automorphisms}}

\medskip

In this section we show how a general smooth automorphism of the Flip orbifold transforms unbounded traces (Theorem \ref{mainB}). Thereby we prove that when it induces the identity on $K_0$ of the Flip orbifold it transforms AC projections to ones that are centrally equivalent to them (Theorem \ref{mainC}).
\medskip

In the proof below we shall make use of the toral automorphisms $\gamma_1, \gamma_2, \gamma_3$ of the rotation algebra, which are also automorphisms of the Flip orbifold (since they commute with the Flip), given by 
\begin{align}\label{gammas}
\gamma_1(U) &= -U, \quad \gamma_1(V)=V	\notag
\\
\gamma_2(U) &= U, \quad \ \ \ \gamma_2(V)=-V
\\
\gamma_3(U) &= -U, \quad \gamma_3(V)=-V		\notag
\end{align}
where $\gamma_3=\gamma_1\gamma_2$. (It's easy to check that these, aside from the identity, are the only toral automorphisms, arising from the usual action of the torus $\mathbb T^2,$ that commute with the Flip.) They transform the unbounded traces according to
\begin{equation}\label{gammaphis}
\phi_{jk} \gamma_1 = (-1)^j \phi_{jk}, \qquad 
\phi_{jk} \gamma_2 = (-1)^k \phi_{jk}, \qquad
\phi_{jk} \gamma_3 = (-1)^{j+k} \phi_{jk}.
\end{equation}

\bigskip

\begin{thm}\label{betaspan}
Let $\beta$ be an automorphism of the smooth Flip orbifold $A_\theta^{\infty,\Phi}$. One has the linear combination
\[
\phi_{jk}\beta = a\phi_{00} + b\phi_{01} +  c\phi_{10} + d \phi_{11} 
\]
where $b,c,d$ are half-integers. Further, $a = 1$ when $jk=00,$ and $a=0$ otherwise.
\end{thm}
\begin{proof} Fix $j,k$. From the paragraph just before Corollary \ref{coro}, since $\phi_{jk}\beta$ is a trace map on $A_\theta^{\infty,\Phi},$ it is a linear combination
\begin{equation}\label{phibetacomb}
\phi_{jk}\beta = K \tau + a \phi_{00} + b\phi_{01} +  c\phi_{10} + d \phi_{11} 
\end{equation}
where $K,a,b,c,d$ are complex constants (depending on $j,k$). 

By Lemma 3.2 in \ccite{SamHouston2018}, there are Flip-invariant Powers-Rieffel projections $Q$ with Connes-Chern character $\T(Q) = (\theta'; \tfrac12, \tfrac12, \tfrac12, \tfrac12)$ for a dense set of irrationals $\theta' \in \mathbb (Z + \mathbb Z\theta) \cap (0,\tfrac12)$. Pick such a $Q$ with $\theta'$ in the range $\tfrac13 < \theta' < \tfrac49 < \tfrac12,$ so that $0 < 3(3\theta'-1) < 1$. 

Applying the toral automorphisms $\gamma_j$ to $Q$ gives the following Connes-Chern characters (in view of \eqref{gammaphis})
\begin{align}\label{gamma1Q}
\T(Q) & = (\theta'; \tfrac12, \tfrac12, \tfrac12, \tfrac12)		\notag
\\
\T(\gamma_1Q) &= (\theta'; \tfrac12,  \tfrac12, -\tfrac12, -\tfrac12)
\\
\T(\gamma_2 Q) &= (\theta'; \tfrac12, - \tfrac12, \tfrac12, - \tfrac12)	\notag
\\
\T(\gamma_3 Q) &= (\theta'; \tfrac12, - \tfrac12, - \tfrac12, \tfrac12) 	\notag
\end{align}
By Theorem 1.2 of \ccite{SWmathscand}, which requires the condition $0 < 3(3\theta'-1) < 1,$ there exists a Flip invariant projection $f_0$ such that $\T(f_0) = (2(3\theta'-1); 0,0,0,0)$. Now we evaluate equation \eqref{phibetacomb} at the projections $1, f_0, Q, \gamma_1Q, \gamma_2Q, \gamma_3Q,$ and recall that $\phi_{jk}$ is always a half-integer on projections, we obtain the respective equations
\[
K + a = \phi_{jk}(1) \in \{0,1\}
\]
\[
12K\theta' - 4K = N_1
\]
\[
K\theta' + \tfrac12 a  + \tfrac12 b +  \tfrac12 c + \tfrac12 d = \tfrac12 N_2
\]
\[
K\theta' + \tfrac12 a + \tfrac12 b - \tfrac12 c - \tfrac12 d = \tfrac12 N_3
\]
\[
K\theta' + \tfrac12 a  - \tfrac12 b + \tfrac12 c - \tfrac12 d = \tfrac12 N_4
\]
\[
K\theta' + \tfrac12 a  - \tfrac12 b - \tfrac12 c + \tfrac12 d = \tfrac12 N_5
\]
for some integers $N_1,\dots,N_5$. Adding the last four of these equations gives
\[
4K\theta' + 2a = \tfrac12 M
\]
where $M = N_2 + N_3 + N_4 + N_5$ is an integer. From $a = \phi_{jk}(1) - K$ we get 
\[
4K\theta' - 2K  = \tfrac12 M - 2\phi_{jk}(1).
\]
But since also $12K\theta' - 4K = N_1$ is an integer, after eliminating $K\theta'$ using these two equations, we see that $K$ must be rational. So from $4K(3\theta' - 1) = N_1$ it follows that $K=0$ (or else $\theta'$ would be rational). Therefore $K = N_1 = 0$ and hence also $a=\phi_{jk}(1)$ and $M = 4\phi_{jk}(1)$. Further, the linear equations above involving $a,b,c,d$ become
\begin{align*} 
b +  c +  d &=  N_2 - a
\\
b -  c -  d &=  N_3 - a
\\
-  b +  c -  d &=  N_4 - a
\\
-  b -  c +  d &=  N_5 - a
\end{align*}
which imply that $b,c,d$ are half integers since $a=\phi_{jk}(1)$ is 1 for $jk=00$ and is 0 otherwise. Equation \eqref{phibetacomb} has now become
\begin{equation}
\phi_{jk}\beta = a \phi_{00} + b\phi_{01} +  c\phi_{10} + d \phi_{11} 
\end{equation}
where $a=\phi_{jk}(1)$ and $b,c,d\in \tfrac12\mathbb Z,$ hence the result.
\end{proof}

\medskip

\begin{qst} As noted above, the invariance relation $\phi_{00}\alpha = \phi_{00}$ holds for any modular $\alpha$ as well as for $\gamma_j$'s. Is it true for all smooth automorphisms $\alpha$ of the Flip orbifold? Or for any smooth automorphism of $A_\theta$  commuting with the Flip?
\end{qst}

\medskip

\begin{cor}
A smooth automorphism of the Flip orbifold that induces the identity map on $K_0(A_\theta^{\infty,\Phi}) = \mathbb Z^6$ necessarily induces the identity map on the zeroth cyclic cohomology group $HC^0(A_\theta^{\infty,\Phi}) = \mathbb C^5$.
\end{cor}

\noindent(Automorphisms homotopic to the identity induce the identity on $K_0$.)

\begin{proof}
Let $\beta$ denote a smooth automorphism for which, by hypothesis, $[\beta(P)] = [P]$ for each projection $P$. In particular, $\phi_{jk}\beta(P) = \phi_{jk}(P)$ for each projection $P,$ and we want to show that $\phi_{jk}\beta = \phi_{jk}$ as maps on the orbifold $A_\theta^{\infty,\Phi}$.\footnote{Although the linear span of projections is norm dense we cannot conclude $\phi_{jk}\beta = \phi_{jk}$ from this since $\phi_{jk}$ are not norm continuous.} From Theorem \ref{betaspan} we have 
\[
\begin{bmatrix} \phi_{00}\beta \\ \phi_{01}\beta \\ \phi_{10}\beta \\ \phi_{11}\beta
\end{bmatrix}
= M \begin{bmatrix} \phi_{00} \\ \phi_{01} \\ \phi_{10} \\ \phi_{11} \end{bmatrix} 
\]
for some $4\times4$ matrix $M$ (with half-integer entries). Evaluating this at the projections $P$ mentioned in \eqref{gamma1Q}, and using $\phi_{jk}\beta(P) = \phi_{jk}(P),$ we get
\[
\begin{bmatrix} \phi_{00}(P) \\ \phi_{01}(P) \\ \phi_{10}(P) \\ \phi_{11}(P)
\end{bmatrix}
= M \begin{bmatrix} \phi_{00}(P) \\ \phi_{01}(P) \\ \phi_{10}(P) \\ \phi_{11}(P) \end{bmatrix}. 
\]
Stacking the left and right sides for the four projections $P = Q, \gamma_1Q, \gamma_2Q, \gamma_3Q$ (and multiplying by a factor of 2) we get
\[
\begin{bmatrix} 1 & 1 & 1 & 1 \\ 1 & 1 & -1 & -1 \\ 1 & -1 & 1 & -1 \\ 1 & -1 & -1 & 1 
\end{bmatrix}
= M \begin{bmatrix} 1 & 1 & 1 & 1 \\ 1 & 1 & -1 & -1 \\ 1 & -1 & 1 & -1 \\ 1 & -1 & -1 & 1 
\end{bmatrix}.
\]
The matrix appearing on the left and right side has determinant $-16$, so is invertible, hence $M$ is the identity matrix, showing that $\phi_{jk}\beta = \phi_{jk}$ as maps on the smooth orbifold.
\end{proof}

\medskip

This, together with the same proof argument used in Section 2, gives us the following result.

\begin{thm}
Let $\beta$ be a smooth automorphism of the Flip orbifold that induces the identity on $K_0(A_\theta^{\infty,\Phi})$. Then $\beta(e) \cequiv e$ for any approximately central projection $e$ in $A_\theta^{\infty,\Phi}$. 
\end{thm}

\smallskip

\begin{rmk}
The smooth automorphism $\gamma_1$ (see \eqref{gammas}) is homotopic to the identity on the rotation algebra $A_\theta,$ but it is not homotopic to the identity on the Flip orbifold since $\gamma_1$ does not induce the identity map on $K_0(A_\theta^\Phi)$ (which is clear from \eqref{gamma1Q}). To show that $\gamma_1(e)$ and $e$ are not centrally equivalent for some AC projection $e$ in the Flip orbifold is not trivial, but can be shown from \ccite{SWb} for an AC Powers-Rieffel projection $e$  (where the same is also shown to occur for the Fourier and Cubic transforms on the Flip orbifold).
\end{rmk}

\begin{rmk}
From Theorem \ref{betaspan}, one obtains a group representation 
\[
M: \text{Aut}(A_\theta^{\infty,\Phi}) \to \text{SL}_{\pm}(3,\mathbb R)
\]
into the modular subgroup $\text{SL}_{\pm}(3,\mathbb R)$ of $\text{GL}(3,\mathbb R)$ consisting of matrices with determinant $\pm1,$ such that the $3\times3$ invertible matrix $M(\beta)$ has half-integer entries for each $\beta,$ defined by noting that any given smooth automorphism $\beta$ (of the Flip orbifold) leaves the subspace spanned by $\phi_{01}, \phi_{10}, \phi_{11}$ invariant.  (In fact, all its powers $M(\beta)^n = M(\beta^n)$ also have half-integer entries for each $n\in \mathbb Z,$ from which it is easy to see that $\det\, M(\beta) = \pm1$.) Do the entries all have to be integers? We do not have an example where $M(\beta)$ has non-integer entries (for smooth $\beta$). For instance, for automorphisms $\alpha_1, \alpha_2$ mentioned in Lemma \ref{lemma} (for even $a,d,$ and for even $a$ and odd $d,$ respectively), one has cyclic permutation matrices
\[
M(\alpha_1) = \begin{bmatrix}
0 & 1 & 0 \\
1 & 0 & 0 \\
0 & 0 & 1 
\end{bmatrix}, \qquad 
M(\alpha_2) = \begin{bmatrix}
0 & 1 & 0 \\
0 & 0 & 1 \\
1 & 0& 0 
\end{bmatrix}.
\]
(The first one has determinant $-1,$ so $M(\beta)$ can have negative determinant.) 
For $\gamma_1$ it is easy to see that its matrix is $M(\gamma_1)= $ diag$(1,-1,-1)$. It would be interesting to know the homomorphic image of $M$ (which contains the subgroup generated by the permutation matrices and the indicated diagonals).
\end{rmk}

\end{document}